\documentclass[reqno]{amsart}
\usepackage{graphicx}
\usepackage{amsfonts,amsmath,amsthm,amssymb,latexsym, cite}%
\usepackage[colorlinks,plainpages,citecolor=magenta, linkcolor=blue, bookmarksnumbered, draft=false]{hyperref}
\usepackage{comment}
\newtheorem{thm}{Theorem}
\newtheorem{cor}{Corollary}
\newtheorem{lem}{Lemma}

\theoremstyle{definition}

\theoremstyle{remark}

\numberwithin{equation}{section}

\newcommand{\dmn}{\mathop{\rm dom}}

\newcommand{\supp}{\mathop{\rm supp}}

\renewcommand{\kappa}{\varkappa}

\newcommand{\Real}{\mathbb R}

\newcommand{\eps}{\varepsilon}

\newcommand{\cH}{\mathcal{H}}

\newcommand{\cR}{\mathcal{R}}

\begin{document}

\begingroup
\renewcommand\thefootnote{}\footnotetext{$^{*}$Corresponding author}
\endgroup

\title[Transmission resonances in scattering by $\delta'$-like combs]{Transmission resonances in scattering\\ by $\delta'$-like combs}
\author{Yuriy Golovaty} \address[YG]{Ivan Franko National University of Lviv,
	1 Universytetska st., 79602 Lviv, Ukraine}%
\email{yuriy.golovaty@lnu.edu.ua}
\author{Rostyslav Hryniv$^{*}$}%
\address[RH]{
	Ukrainian Catholic University, 2a Kozelnytska str., 79026, Lviv, Ukraine \and
	University of Rzesz\'{o}w, 1 Pigonia str., 35-310 Rzesz\'{o}w, Poland}%
\email{rhryniv@ucu.edu.ua, rhryniv@ur.edu.pl}

\author{Stanislav Lavrynenko}%
\address[SL]{Ivan Franko National University of Lviv,
	1 Universytetska st., 79602 Lviv, Ukraine}
\email{stanislav.lavrynenko@lnu.edu.ua}%

\subjclass[2020]{81U15, 47A40, 34L25}%

\keywords{Exactly solvable models, quantum mechanics, point interactions, $\delta'$-potentials, particle scattering, transmission probability}

\date{1 July 2025}%
\begin{abstract}
    We introduce a new exactly solvable model in quantum mechanics that describes the propagation of particles through a potential field created by regularly spaced $\delta'$-type point interactions, which model the localized dipoles often observed in crystal structures. We refer to the corresponding potentials as $\delta'_\theta$-combs, where the parameter $\theta$ represents the contrast of the resonant wave at zero energy and determines the interface conditions in the Hamiltonians. We explicitly calculate the scattering matrix for these systems and prove that the transmission probability exhibits sharp resonance peaks while rapidly decaying at other frequencies. Consequently, Hamiltonians with $\delta'_\theta$-comb potentials act as quantum filters, permitting tunnelling only for specific wave frequencies. 
    Furthermore, for each $\theta > 0$, we construct a family of regularized Hamiltonians approximating the ideal model and prove that their transmission probabilities have a similar structure, thereby confirming the physical realizability of the band-pass filtering effect.
\end{abstract}
\maketitle
\section{Introduction}

In this paper, we study quantum tunnelling effects in Hamiltonians with potentials that we call $\delta_\theta'$-combs, and the principal results are illustrated in Fig.~\ref{FigStart}. These potentials serve as exactly solvable models for electron transport in crystal structures and consist of multiple equally spaced, scale-invariant point interactions naturally associated with $\delta'$-distributions, i.e., the distributional derivatives of Dirac's delta function. Fig.~\ref{FigStart} illustrates the dependence of the transmission probability~$T$ of a quantum particle on its wave frequency~$k$; notably, $\delta'_\theta$-combs produce sharp quantum filters permitting particle transmission only within specific narrow energy bands while suppressing tunnelling at all other energies.

\begin{figure}[!h]
  \centering
  \includegraphics[scale=.38]{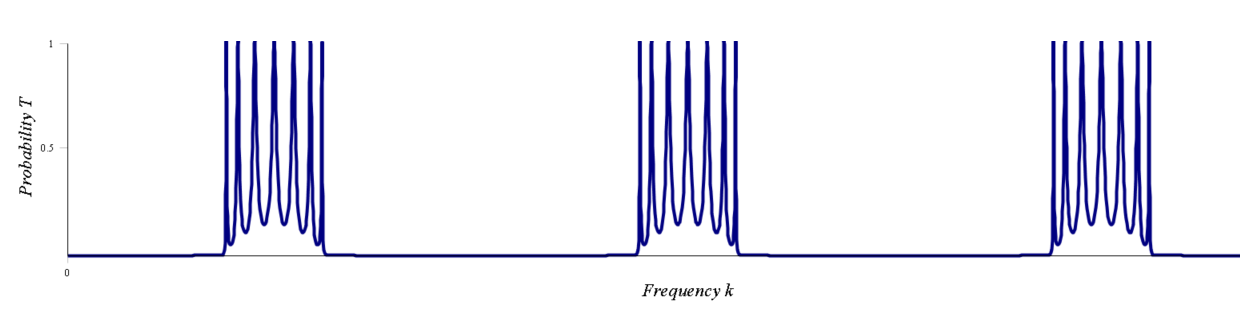}\\
  \caption{Peaks in the probability of tunneling through the $\delta'_\theta$-comb.}\label{FigStart}
\end{figure}

The scale-invariant point interactions in the constructed $\delta_\theta'$-combs effectively model non-vanishing dipole moments caused by the asymmetric distribution of po\-si\-tive and negative charges in certain crystal structures. Such dipole moments lead to important physical properties, including spontaneous polarization, piezoelectricity, pyroelectricity, and ferroelectricity. For instance, Fig.~\ref{FigZnO} illustrates the crystallization of zinc and oxygen into a wurtzite lattice~\cite{ZnO}, typical of binary compounds such as ZnO. The structure consists of two interpenetrating hexagonal close-packed sublattices---one for each type of atom (Zn and O)---and the relative displacements of these sublattices induce local electric dipole moments.  

Particle transport along the $x$-axis in such media involves scattering from a sequence of these dipole-induced potentials, which can be effectively modeled by \mbox{$\delta'$-like} point interactions. Moreover, under certain conditions, hexagonal graphite-like layers can be extracted from such crystals, and these may then roll into cylindrical nanotubes~\cite{ZnO}. The transport of particles through such nanotubes is effectively one-di\-men\-si\-o\-nal, and the influence of local asymmetry-caused dipoles can be approximated by a finite array of equally spaced $\delta'$-like fields. This motivates the study of the scattering problem for periodic $\delta'_\theta$-combs, which serve as an analytically tractable model for electron transport in such nanostructures.

\begin{figure}[t]
  \centering
  \includegraphics[scale=.3]{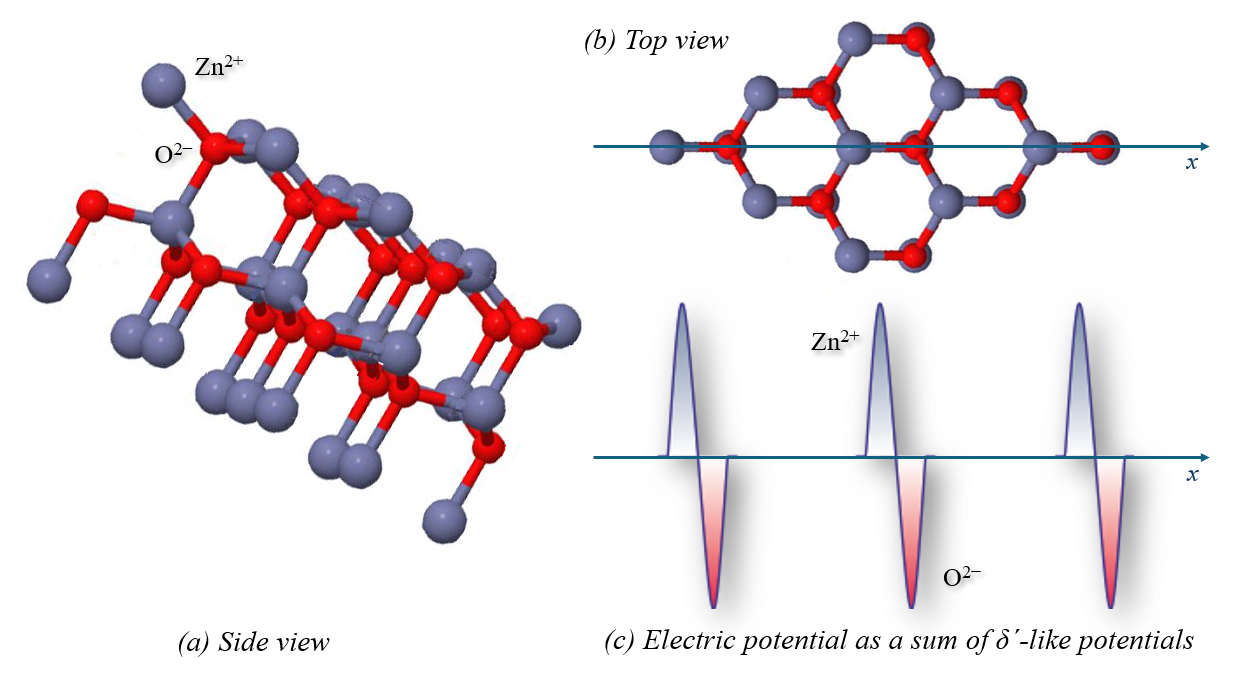}\\
  \caption{Crystal structure of zinc oxide ZnO with non-vanishing local dipole moments.}\label{FigZnO}
\end{figure}

To further explain the motivation and set up the notation, we start with the case of a single $\delta_\theta'$-potential. We denote by $S_\theta$ the free Schr\"{o}dinger operator with point interaction at $x=a$ acting on the functions in $W_2^2(\mathbb{R}\setminus\{a\})$ subject to the interface conditions
\begin{equation}\label{PointInteractionX0}
  \psi(a+0) = \theta \psi(a-0), \quad \psi'(a+0) = \theta^{-1} \psi'(a-0)
\end{equation}
for some $\theta\in\Real\setminus\{0\}$. As we explain below, $S_\theta$ can be obtained as the norm resolvent limit of the regular Schr\"odinger operators with scaled $\delta'$-type resonant potentials, and the parameter $\theta$ is determined by the corresponding half-bound state. For that reason, we call the operator $S_\theta$ the \textit{Schr\"odinger operator with $\delta'_\theta$ potential} and denote it
\begin{equation}\label{Stheta}
        S_\theta = - \frac{d^2}{dx^2} + \delta'_\theta(x-a). 
\end{equation}

From the point of view of quantum mechanics, the $\delta'$-type potentials can be thought of as an idealized model for localized dipoles, i.e., narrow sharp potentials combining a high potential barrier followed by a deep potential well as in Fig.~\ref{FigZnO}\,(c).
A Hamiltonian with such an idealized dipole can be modelled as the limit of the Schr\"odinger operators
\begin{equation}\label{eq:H_eps}
    H_\eps=-\frac{d^2}{dx^2}+\eps^{-2}V(\eps^{-1}(x-a))
\end{equation}
as $\eps\to0$, where $V$ is an arbitrary short-range potential. Indeed, if $\int_{\Real}V(x)\,dx=0$ and $\int_{\Real}xV(x)\,dx=-1$, then the potentials $\eps^{-2}V(\eps^{-1}(x-a))$ converge to $\delta'(x-a)$ in the sense of distributions as $\eps$ tends to $0$.

The arguments of~\cite{SebRMP} demonstrated that the operators $H_\eps$ should converge to the direct sum $H_\mathrm{D}$ of two free half-line Schr\"odinger operators subject to the Dirichlet condition at $x=a$. From the quantum particle scattering perspective, that result would mean that the resulting $\delta'$-type potentials should split the quantum system into two non-interacting parts, i.e., that a localized dipole should be impenetrable to quantum particles. However, a later study \cite{ChristianZolotarIermak03} suggested an example contradicting that property and described resonances for the transmission probability through a piecewise constant $\delta'$-like potential. More generally, it was proved in~\cite{GolovatyHryniv2010, GolovatyHrynivProcEdinburgh2013} that for an arbitrary real-valued potential $V$ of the Marchenko class~$L_2(\Real,1+|x|)$ the operators~$H_\eps$ converge in the norm resolvent topology as $\eps\to0$, and that the limit is~$H_\mathrm{D}$ if~$V$ is non-resonant and $S_\theta$ if~$V$ is resonant.

Here, we call a potential~$V$ \textit{resonant}, or having a \textit{zero-energy resonance} if the equation $-u''+Vu=0$ has a solution \textit{(half-bound state)} $u$ that is bounded on the whole line. This half-bound state has then non-zero limits $u(\pm\infty)$ at both infinities, and the parameter~$\theta$ represents the contrast of~$u$ and is equal to the ratio of these limits,
\begin{equation}\label{Theta}
  \theta=\frac{u(+\infty)}{u(-\infty)}.
\end{equation}
This fact justifies the identification of~$S_\theta$ as the Schr\"odinger operator with potential~$\delta'_\theta(x-a)$, where the subscript~$\theta$ refers to the interface conditions~\eqref{PointInteractionX0}.
We should stress, however, that there is no explicit dependence of $\theta$ on the shape~$V$ in the family of potentials~$V_\eps(x) = \eps^{-2}V(\eps^{-1}(x-a))$: the Hamiltonian $H_\eps$ may converge to some $S_\theta$ even if  $V_\eps$ have no limit in the distributional sense, and two resonant families of $V_\eps$ converging to the same distribution $\beta\delta'$ may correspond to different limiting Hamiltonians $S_\theta$.

In \cite{GolovatyManko2009}, the authors studied the scattering properties of the family $H_\eps$ as $\eps\to0$. It was proved therein that the transmission probability $T_\eps(k)$ of $H_\eps$ vanishes as $\eps\to 0$ if $V$ is non-resonant (in agreement with~\cite{SebRMP}); however, if $V$ has a zero-energy resonance with a half-bound state $u$, then this probability converges to the non-zero value
\begin{equation}\label{T1}
  T(k)=\frac{4\theta^2}{(\theta^2+1)^2}
\end{equation}
for all energies $E=k^2$, with $\theta$ given by \eqref{Theta}.
It is worth noting that resonant potentials are not uncommon. For example, for any integrable function $V$ that vanishes sufficiently quickly at infinity, there is a countable set of constants $c_n$ such that the potentials $c_n V$ have a zero-energy resonance (see \cite{GolovatyManko2009} for details).

The main results of this paper reveal striking properties of the transmission probability $T_n(\theta,k)$ for the Hamiltonian 
\begin{equation*}
    \mathcal{H} = -\frac{d^2}{dx^2}+ \sum_{j=0}^{n-1} \delta'_{\theta}(x - jh),
\end{equation*}
whose potential is a $\delta'_\theta$-comb consisting of $n$ point interactions~$\delta'_\theta$ placed at $a_j = jh$, $j=0,1,\dots,n-1$.
In Section~\ref{sec:transmission-resonances}, we demonstrate that this Hamiltonian defines a new exactly solvable quantum mechanical model, in the sense that its scattering coefficients and transmission probability can be expressed explicitly as analytic functions of the interface para\-me\-ter~$\theta$ and the wave frequency~$k$. 
It turns out that this model exhibits strong band-pass filtering and tunnelling effects through potentials with structured point interactions. 
Namely, for each $\theta$ not equal to $0$ or $1$, we prove that $T_n$ is $\pi$-periodic with respect to $k$ and possesses exactly $n-1$ transmission resonances over the range $k \in (0,\pi)$, at which $T_n=1$. These resonances are explicitly determined in terms of $n$ and $\theta$ and are shown to converge to the midpoint $k=\pi/2$ as $\theta\to0$. Moreover, for values of $k$ not equal to $\pi/2$, the transmission probability $T_n(\theta,k)$ decays as $O(\theta^{2n})$, demonstrating that the $\delta'_\theta$-comb functions as a highly efficient band-pass frequency filter, cf.~Fig.~\ref{FigStart}.

To the best of our knowledge, no exactly solvable quantum mechanical models have been reported that exhibit the same combination of transmission tunnelling and filtering properties as the one proposed here. The classic Kronig--Penney model~\cite{KronigPenney}, introduced in 1931 to describe electron transport in periodic crystals, employed a periodic array of local potentials, which converge to a Dirac $\delta$ function in the limit. This was seemingly among the first non-trivial exactly solvable models of quantum mechanics. A rigorous mathematical treatment of such models, including an extensive review of related results up to the early 2000s, can be found in~\cite{AGHH2005,AK2000}. These works systematically analyze Schrödinger operators with point interactions, particularly of $\delta$- and $\delta'$-type. It is important to note that the $\delta'$-interactions studied in these references differ from the $\delta'_\theta$-potentials considered in this paper. Notably, the work~\cite{Mostafazadeh2011} offers a thorough analysis of spectral singularities occurring in the general four-parameter family of complex point interactions, with particular attention to $\mathcal{P}$-, $\mathcal{T}$-, and $\mathcal{PT}$-symmetric cases, as well as the coalescence of singularities and bound states.

For models involving multiple regularly spaced $\delta$-potentials---often referred to as Dirac $\delta$-combs---the transmission probability was explicitly derived as early as 1974 in~\cite{Rorres1974}, and many similar results have been reported since (see, e.g.,~\cite{Kiang1974,Cvetic1981,Erdos1982,Martin2014,Erman2018}). A systematic treatment of wave propagation in structured periodic media, where a localized potential is repeated $n$ times, is presented in~\cite{Griffiths01}. That work derives fundamental relations for scattering coefficients, and, in particular, discusses the case of electromagnetic scattering by $\delta$-comb potentials and demonstrates the characteristic behaviour of the transmission probability in the Kronig--Penney limiting case~\cite{KronigPenney} as $n\to\infty$.

Recently, in~\cite{Zol25}, the authors investigated the periodic behaviour of transmission probability for a square quantum well. However, this periodicity was not in terms of energy $E = k^2$ or wave frequency $k$, but rather with respect to the thickness of the quantum well. For a fixed energy $E = k_0^2$, the potential shape $V$ in~\eqref{eq:H_eps} was designed so that, for all small~$\eps$, the transmission probability $T_\varepsilon(k)$ has a peak at $k=k_0$. The authors further proposed placing this quantum well outside the background potential to function as a wave filter at a specific frequency~$k_0$.

In contrast, the model introduced in this paper provides a broader and more adaptable filtering mechanism, acting as a band-pass filter over a range of frequencies rather than a single resonance peak. This distinction makes our model significantly more versatile for applications where selective transmission over multiple energy bands is required. Moreover, as demonstrated in Section~\ref{sec:example}, Hamiltonians~\eqref{Heps} with appropriately regularized $\delta'$-type locally periodic potentials, which approximate~$\cH$ in the limit $\eps\to0$, also exhibit similar transmission resonances and filtering properties.

The rest of the paper is organized as follows. In Section~\ref{sec:setting}, we introduce Hamiltonians with $\delta'_\theta$-comb potentials and discuss the associated scattering solutions and scattering coefficients. Section~\ref{sec:transfer-matrix} derives the transfer matrix~$M_n(\theta,k)$ and establishes general properties of the transmission probability that follow directly from the structure of~$M_n(\theta,k)$.  In Section~\ref{sec:transmission-resonances}, we present our main analytical results: we first compute $M_n(\theta,k)$ for small values of $n$ ($n=2,3,4$) and analyze the corresponding transmission probabilities $T_n(\theta,k)$. We then derive closed-form expressions for $M_n(\theta,k)$ and $T_n(\theta,k)$ for arbitrary $n$ and demonstrate two key effects---the tunnelling phenomenon ($T_n = 1$) and the narrow-band filtering behavior of $\delta'_\theta$-comb potentials. 
In Section~\ref{sec:example}, we construct, for every $\theta>0$, a resonant potential~$V$ with this parameter $\theta$ for the half-bound state, and then prove that the transmission probabilities for the scaled Hamiltonian~$\cH_\eps$ in~\eqref{Heps} converge to those of $\cH$. This confirms that the filtering and tunnelling effects persist in these physically realizable approximations. 


\section{Setting of the problem}
\label{sec:setting}

Let us consider the Hamiltonian
\begin{equation}\label{Heps}
   \cH_\eps=-\frac{d^2}{dx^2}+\frac1{\eps^2}\sum_{j=0}^{n-1}V\left(\frac{x-jh}\eps\right),
\end{equation}
where $V$ is an integrable function of compact support, $h>0$ is the spacing between interaction centers, and $\eps>0$ is a small scaling parameter. Such a Hamil\-to\-nian describes the interaction of particles with localized potentials concentrated near the points of the set $X_h=\{jh\}_{j=0}^{n-1}$; for small $\eps$, these potentials mimick $\delta'$-type dipole behaviour.

Assume that $V$ exhibits a zero-energy resonance (in the sense explained in the Introduction), with associated parameter $\theta$ defined via~\eqref{Theta}. The results of~\cite{GolovatyHryniv2010} imply that, as 
$\eps\to0$, the operators $\cH_\eps$ converge in the norm resolvent sense to the limiting operator formally expressed as 
\begin{equation}\label{OperatorH0} 
    \cH = -\frac{d^2}{dx^2} + \sum_{j=0}^{n-1} \delta'_{\theta}(x - jh).
\end{equation} 
Here, each point interaction~$\delta'_\theta(x - jh)$ imposes an interface condition of the form~\eqref{PointInteractionX0} at $a= jh$.
More precisely, the limiting operator~$\cH$ acts as
\begin{equation}\label{eq:cHdef}
    \cH\psi=-\psi''    
\end{equation}
on the domain
\begin{equation}\label{eq:cHinterface}
\begin{aligned}
  \dmn \cH =\big\{\psi\in W_2^2(\Real\setminus X_h)\colon  &\psi(x+0) = \theta \psi(x-0), \\ &\psi'(x+0) = \theta^{-1} \psi'(x-0)\text{ for all } x\in X_h\big\}.
\end{aligned}
\end{equation}
We refer to~\eqref{OperatorH0} as the Hamiltonian with a $\delta'_\theta$-\textit{comb potential}. The main aim of this paper is to analyze the scattering properties of such Hamiltonians and to demonstrate the unusual behaviour of their transmission probability functions.

For $k\in\Real$, we consider a left scattering solution $\psi_-$ of $\cH\psi=k^2\psi$ such that
\begin{equation*}
 \psi_-(x)  = \begin{cases}
     e^{ikx}+r_-(k)e^{-ikx} \quad & \text{for }x<0,\\
     t_-(k)e^{ikx} \quad &\text{for }x>(n-1)h.
 \end{cases}
\end{equation*}
The coefficients $r_-(k)$ and $t_-(k)$ are called the \textit{left reflection} and \textit{transmission amplitudes} respectively and are unknown here. The solution $\psi_-$ describes the scattering of an incoming flux of particles moving left to right. If the particles are incident from the right, the scattering solution is then
\begin{equation*}
  \psi_+(x)  = \begin{cases}
     t_+(k)e^{ikx} \quad & \text{for }x<0,\\
     e^{ikx}+r_+(k)e^{-ikx} \quad &\text{for }x>(n-1)h
     \end{cases}
\end{equation*}
with the \textit{right reflection} coefficient $r_+(k)$ and \textit{right transmission} coefficient $t_+(k)$. It is known that $t_-(k)=t_+(k)$, $|r_-(k)| = |r_+(k)|$, and $|t_\pm(k)|^2 + |r_\pm(k)|^2 = 1$ for all $k\in \Real$. Our main goal is to study the \textit{transmission probability} $T(k)=|t_-(k)|^2$.

Note that it is sufficient to study this model only for the case when the point interactions in the $\delta'_\theta$-comb are arranged with the step size~$h=1$. In fact, changing the variable $x$ to $h^{-1}x$ preserves the scale-invariant $\delta'_\theta$ interface conditions~\eqref{PointInteractionX0} but positions them at points of the set $X_1$; simultaneously, the energy $k^2$ scales to $h^{-2}k^2$, and the Schr\"{o}dinger equation becomes $\cH\psi=h^{-2}k^2\psi$. As a result, scaling the coordinates leads to scaling the energy but does not change the model itself.

Therefore, for given $n\in\mathbb{N}$ and $\theta\in\Real\setminus\{0\}$, we consider the problem
\begin{gather}\label{ScattProblemEq}
  \psi''+k^2\psi=0,\qquad k\in\Real, \quad x\in\Real\setminus\{0,\dots,n-1\}, \\\label{ScattProblemPI}
  \psi(j+0)=\theta\psi(j-0), \quad\psi'(j+0)=\theta^{-1}\psi'(j-0),\quad j=0,\dots,n-1,
\end{gather}
corresponding to~\eqref{OperatorH0} with $h=1$, and we look for the left scattering solution
\begin{equation}\label{ScattProblAsm}
   \psi_-(x)  = \begin{cases}
     e^{ikx}+r_n(k)e^{-ikx} \quad & \text{for }x<0,\\
     t_n(k)e^{ikx} \quad &\text{for }x>(n-1)h.
     \end{cases}
\end{equation}
We simplify the notations for the left reflection and transmission coefficients by omitting the minus sign in the subscript but instead explicitly indicate the number~$n$ of the interactions in the $\delta'_\theta$-comb. Our aim is to study the dependence of the transmission  probability
\begin{equation*}
  T_n(\theta,k)=|t_n(\theta,k)|^2
\end{equation*}
on $\theta$, $k$, and $n$. 
Note that the probability $T_1(\theta,k)$ is given by \eqref{T1} and is independent of~$k$; also, the point interactions \eqref{ScattProblemPI} are trivial for $\theta=1$, so $T_n(1,k)=1$ for all $k$.

\section{Properties of the transmission probability}
\label{sec:transfer-matrix}
Quantum mechanical tunnelling through the $\delta'_\theta$-comb is typically analyzed using the transfer matrix method. A comprehensive and rigorous treatment of transfer matrices, including the derivation of their main properties such as the composition rule, can be found in the review~\cite{Mostafazadeh2020}. In our setting, all computations are relatively straightforward and explicit, so we present the construction in full detail for completeness.

Each solution of the equation $\psi''+k^2\psi=0$ on the set $\Real\setminus\{a\}$ can be written as
\begin{equation*}
  \psi(x)=
  \begin{cases}
    \alpha_1e^{ikx}+\alpha_2e^{-ikx}&\text{for }x<a,\\
    \beta_1e^{ikx}+\beta_2e^{-ikx}&\text{for }x>a,
  \end{cases}
\end{equation*}
where $\vec{\alpha}=[\alpha_1,\alpha_2]^\top$  and $\vec{\beta}=[\beta_1,\beta_2]^\top$ are complex vectors.
If $\psi$ satisfies the interface conditions~\eqref{PointInteractionX0}, then $\vec{\alpha}$  and $\vec{\beta}$ are related by the linear transformation $\vec{\beta}=M(\theta, ak)\vec{\alpha}$,
where
\begin{equation}\label{TransferMatrix1}
  M(\theta, z)=\dfrac1{2\theta}
  \begin{bmatrix}
\theta^2+1 & (\theta^2-1)e^{-2iz} \\
    (\theta^2-1)e^{2iz} &  \theta^2+1
  \end{bmatrix}.
\end{equation}
The wave components of the left scattering solution~\eqref{ScattProblAsm} to the left and right of the $\delta'_\theta$-comb are similarly related by a matrix $ M_n(\theta,k)$ called the \textit{transfer matrix} for problem \eqref{ScattProblemEq}--\eqref{ScattProblAsm}; namely,
\begin{equation}\label{TMR}
  \begin{bmatrix}
    t_n(\theta,k)\\0
  \end{bmatrix}=
  M_n(\theta,k)\begin{bmatrix}
    1\\r_n(\theta,k)
  \end{bmatrix},
\end{equation}
where
\begin{equation}\label{TransferMatrixN}
  M_n(\theta,k)=M(\theta, (n-1)k)\cdots M(\theta, 2k)\,M(\theta, k)\,M(\theta, 0).
\end{equation}
The transfer matrix belongs to the special group $SU(1,1)$ consisting of all complex matrices of the form
\begin{equation*}
    \begin{bmatrix}
    z_1 & \bar{z}_2 \\
    z_2  &  \bar{z}_1
  \end{bmatrix},\qquad |z_1|^2-|z_2|^2=1.
\end{equation*}
This property together with \eqref{TMR} allows us to express the transfer matrix in terms of the scattering amplitudes:
\begin{equation}\label{TransferMatrixRepr}
  M_n(\theta, k)=
  \begin{bmatrix}
1/\overline{t_n}(\theta,k) & -\overline{r_n}(\theta,k)/\overline{t_n}(\theta,k)\\
  -r_n(\theta,k)/t_n(\theta,k) &  1/t_n(\theta,k)
  \end{bmatrix}.
\end{equation}
Thus, to calculate the transmission probabi\-li\-ty, it suffices to compute the transfer matrix~\eqref{TransferMatrixRepr}. Observe that the fact that $M_n(\theta,k) \in SU(1,1)$ yields the known relation
\begin{equation}\label{eq:scat-identity}
    |r_n(\theta,k)|^2 + |t_n(\theta,k)|^2 = 1;
\end{equation}
also, $t_n(\theta,k)$ vanishes for no $k\in\Real$ as~\eqref{TMR} would mean that $M_n(\theta,k)$ is singular.

The following theorem summarizes the basic properties of the transmission pro\-ba\-bility that can be obtained from the properties of the transfer matrix without explicitly calculating it.

\begin{thm}\label{Theorem1}
    Suppose that $\theta$ is different from $0$ and $1$.
    The transmission probability $T_n(\theta,k)$ has the following properties:
    \begin{itemize}
      \item[\textit{(i)}]  $T_n(\theta,k)$ is a $\pi$-periodic function of $k$;
      \item[\textit{(ii)}] $T_n(\theta,k)=T_n(\theta,\pi-k)$ for all $k\in(0,\pi)$, i.e., the graph of $T_n(\theta,\,\cdot\,)$ on the interval $(0,\pi)$ exhibits symmetry with respect to the line $k=\frac\pi2$.
      \item[\textit{(iii)}] $T_n(\theta,k)$ is invariant under the transformations $\theta\mapsto -\theta$ and $\theta\mapsto \theta^{-1}$, i.e.,
    \begin{equation*}
      T_n(\theta,k)=T_n(-\theta,k)=T_n(\theta^{-1},k)
    \end{equation*}
    for all $k\in \Real$.
    \end{itemize}
\end{thm}

\begin{proof}
    \textit{(i)}
    By construction, the entries of $M_n(\theta, k)$ are trigonometric polynomials of $k$. In particular, for the transmission amplitude we find by induction that
    \begin{equation*}
     \frac1{t_n(\theta,k)}=\sum_{m=0}^{n-1}c_{m,n}(\theta)e^{2imk}, \quad k \in \Real,
    \end{equation*}
    demonstrating $\pi$-periodicity of $T_n(\theta, k)$ in $k$.

    \textit{(ii)} We have $M(\theta, j(\pi-\kappa))=\overline{M(\theta, j\kappa)}$ for integer $j$, and thus $M_n(\theta, \pi-k)=\overline{M_n(\theta,k)}$ due to~\eqref{TransferMatrixN}. The explicit formula~\eqref{TransferMatrixRepr} implies that $t_n(\theta, \pi-k)=\overline{t_n(\theta,k)}$, so that $T_n(\theta,\pi-k)=T_n(\theta,k)$.

    \textit{(iii)} We also have  $M(-\theta, \kappa)=-M(\theta, \kappa)$ and thus $M_n(-\theta, k)=(-1)^{n}M_n(\theta,k)$.
    Again, using representation \eqref{TransferMatrixRepr}, we have $t_n(-\theta,k)=(-1)^{n}t_n(\theta,k)$, and therefore $T_n(-\theta,k)=T_n(\theta,k)$.

    Replacing $\theta$ by $\theta^{-1}$ in point interactions \eqref{ScattProblemPI} gives
    \begin{equation*}
      \psi(j-0)=\theta\psi(j+0), \quad\psi'(j-0)=\theta^{-1}\psi'(j+0),
    \end{equation*}
    so the transformation $\theta\mapsto \theta^{-1}$ is equivalent to a change of direction on the real axis.
    This assertion is further supported by the obvious equalities
    \begin{equation*}
      M(\theta^{-1}, \kappa)=M^{-1}(\theta, \kappa), \qquad M_n(\theta^{-1}, k)=M_n^{-1}(\theta,k).
    \end{equation*}
    Consequently, $T_n(\theta,k)=T_n(\theta^{-1},k)$ is a statement that the transmission probabi\-li\-ty is independent of the incidence direction.
\end{proof}

\begin{figure}[!ht]
  \centering
 \includegraphics[scale=.38]{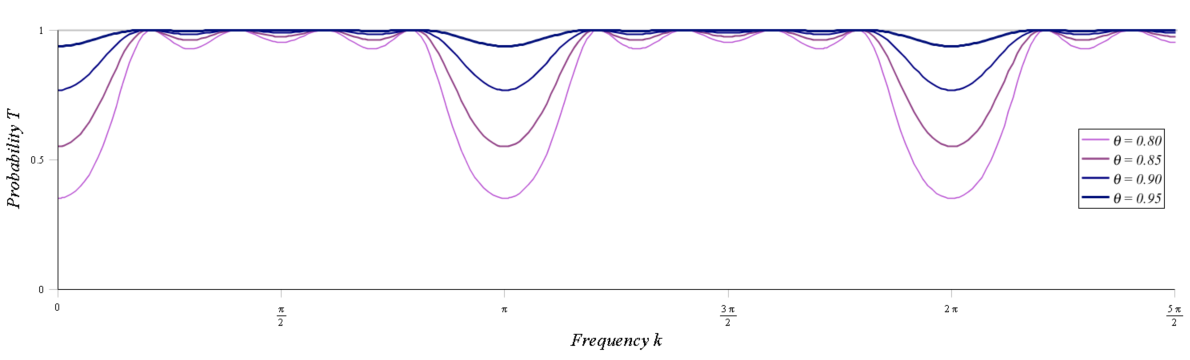}\\
  \caption{Plots of the transmission probability $T_5(\theta,k)$ for several values of $\theta$ close to $1$. }\label{FigThetaGoes1}
\end{figure}

As follows from the theorem, it is sufficient to compute  $T_n(\theta,k)$ for positive $\theta$ less than $1$.
The most interesting scenario arises when $\theta$ approaches zero. In this case, the comb of  $\delta'_\theta$-interactions becomes a filter, passing only particles with a narrow energy range (see Fig.~\ref{FigStart}).
We also note that the $\delta'_\theta$ point interactions disappear at $\theta=1$; for $\theta$ close to~$1$, transmission probability are also close to~$1$; see Fig.~\ref{FigThetaGoes1} and Corollary~\ref{cor:upper_bound}.

The following theorem states that the probability $T_n(\theta,k)$ tends to zero, as $\theta\to 0$, almost everywhere.

\begin{thm}
    Suppose $n\ge 2$ and $\theta\in(0,1)$. Then,  for any $k\in (0,\pi]$ other than $\frac\pi2$, the transmission probability $T_n(\theta,k)$ tends to zero as $\theta\to 0$.
\end{thm}

\begin{proof}
    Since
      \begin{equation*}
     2\theta M(\theta, \kappa)=
      \begin{bmatrix}
       1 & -e^{-2i\kappa} \\
        -e^{2i\kappa} & 1
      \end{bmatrix}+\theta^2 \begin{bmatrix}
       1 & e^{-2i\kappa} \\
        e^{2i\kappa} & 1
      \end{bmatrix},
    \end{equation*}
    we have the formula
    \begin{equation}\label{MnAsymptotics}
      2^n\theta^n M_n(\theta, k)=A_n(k)+\theta^2B_n(\theta, k),
    \end{equation}
    where
    \begin{equation*}
      A_n(k)=\begin{bmatrix}
       1 & \kern-4pt-e^{-2i(n-1)k} \\
    -e^{2i(n-1)k} &\kern-4pt 1
      \end{bmatrix}\cdots\begin{bmatrix}
       1 & \kern-4pt-e^{-4ik} \\
        -e^{4ik} & \kern-4pt1
      \end{bmatrix}\begin{bmatrix}
       1 & \kern-4pt-e^{-2ik} \\
        -e^{2ik} & \kern-4pt1
      \end{bmatrix}\begin{bmatrix}
       1 & \kern-4pt-1 \\
        -1 & \kern-4pt1
      \end{bmatrix}.
    \end{equation*}
    In addition, the norm of the matrix $B_n(\theta, k)$ is uniformly bounded on $(0,1)\times(0,\pi]$.
    Arguing by induction, one proves that
    \begin{equation*}
     A_n(k)=
      \begin{bmatrix}
      \phantom{-}(e^{-2ik}+1)^{n-1} & -(e^{-2ik}+1)^{n-1} \\
        -(e^{2ik}+1)^{n-1} & \;(e^{2ik}+1)^{n-1}
      \end{bmatrix}.
    \end{equation*}
    Using \eqref{MnAsymptotics} along with matrix representation \eqref{TransferMatrixRepr}, we estimate
    \begin{equation*}
      \Big|2^n\theta^n |t_n(\theta,k)|^{-1}-|e^{2ik}+1|^{n-1}\Big|\le c_1\theta^2,
    \end{equation*}
    where $c_1$ depends only on $n$. Since $|e^{2ik}+1|=2|\cos k|$, this inequality can be rewritten as
    \begin{equation*}
      \Big|2\theta^n |t_n(\theta,k)|^{-1}-|\cos k|^{n-1}\Big|\le c_2\theta^2.
    \end{equation*}
    If $k\neq\frac\pi2$, then the ratio $\theta^n/|t_n(\theta,k)|$
    has the strictly positive limit as $\theta\to 0$, and therefore $|t_n(\theta,k)|$ tends to zero.
\end{proof}

The convergence of $T_n(\theta,k)$ described in the theorem is not uniform with respect to $k$. In fact, for every $\theta>0$, there are $n-1$ frequencies $k$ in the range $(0,\pi)$ at which the transmission probability $T_n(\theta,k)$ is equal to $1$, as we demonstrate in the next section.

\section{Transmission resonances in scattering through $\delta'_\theta$-comb}
\label{sec:transmission-resonances}

The values $k$ at which the function $T_n(\theta,\,\cdot\,)$ reaches its local maxima are called \textit{transmission resonances}. The set of all transmission resonances on $(0,\pi)$ is denoted~$\cR_n^\theta$. Based on the results of computer simulations, we hypothesized that the scattering process through the $\delta'_\theta$-comb with $n$ point interactions demonstrates $n-1$ resonant values, i.e.,
\begin{equation*}
  \cR_n^\theta=\big\{k_1(\theta), k_2(\theta),\dots, k_{n-1}(\theta)\big\},
\end{equation*}
and that the probability $T_n(\theta,k)$ at these points is $1$. Below, we demonstrate this for $n=2,3,4$, and then derive the explicit formula for $T_n$ allowing a thorough description of the transmission through the $\delta'_\theta$-comb.

In view of Theorem~\ref{Theorem1},  the set $\cR_n^\theta$ is symmetric with respect to the point $k=\frac\pi2$ and shrinks to this point as $\theta$ tends to zero: $k_{n-j}(\theta)= \pi -k_j(\theta)$ and $k_j(\theta)\to \frac\pi2$, as $\theta\to 0$ for all $j=1,\dots,n-1$.  For even $n$, the set $\cR_n^\theta$ contains an odd number of points, and therefore, one of them must coincide with point $\frac\pi2$, namely $k_{n/2}(\theta)=\frac\pi2$. Furthermore, we observe that
\begin{equation*}
  M\left(\theta,\tfrac{\pi m}2\right)=\dfrac1{2\theta}
  \begin{bmatrix}
\theta^2+1 & (-1)^m (\theta^2-1) \\
    (-1)^m (\theta^2-1)&  \theta^2+1
  \end{bmatrix}
\end{equation*}
for every integer $m$, and $M\big(\theta,\tfrac{\pi (m+1)}2\big) M\big(\theta,\tfrac{\pi m}2\big)=E$,
where $E$ is the identity matrix. Hence, $M_n(\theta,\tfrac{\pi}{2})=E$ because product \eqref{TransferMatrixN} has an even number of factors, and  therefore $T_{n}(\theta,\frac\pi2)=1$. However, this is the only case where the probability $T_n(\theta,k)$ has a ``stationary'' resonance that does not change with $\theta$; other resonances are moving as $\theta\to 0$.

\subsection{Transmission probabilities and resonance sets for a small number of point interactions}
For ease of calculations, we introduce the notation $\theta=e^\kappa$; then
\begin{equation*}
  \cosh\kappa=\tfrac{\theta^2+1}{2\theta}, \qquad  \sinh\kappa=\tfrac{\theta^2-1}{2\theta},
\end{equation*}
and the matrix \eqref{TransferMatrix1} can be written as
\begin{equation*}
   M(\kappa,z) = \begin{bmatrix} \cosh\kappa &\sinh\kappa\: e^{-2iz} \\\sinh\kappa\: e^{2iz} & \cosh\kappa \end{bmatrix}.
\end{equation*}

The potential with two $\delta'_\theta$-centers is the simplest case: then $\cR_2^\theta=\left\{\frac\pi2\right\}$ and $T_2(\theta,\frac\pi2)=1$, as we explained above. The elements of $M_2(\theta,k)$ are of the form
\begin{equation}\label{t2}
  \frac1{t_2(\theta,k)}=e^{2ik}\sinh^2\kappa+\cosh^2\kappa, \qquad
  \frac {r_2(\theta,k)}{t_2(\theta,k)}=-\sinh\kappa\cosh\kappa\;(e^{2ik}+1).
\end{equation}
The fact that the transmission probability
\begin{equation*}
  T_2(\theta,k)=\frac1{4\sinh^2\kappa\cosh^2\kappa\cos^2k+1}=
  \frac{4\theta^4}{(1-\theta^4)^2\cos^2k+4\theta^4}
\end{equation*}
has no other resonances follows from a simple analysis of this function (see Fig.~\ref{FigN2}).

\begin{figure}[!t]
  \centering
  \includegraphics[scale=.38]{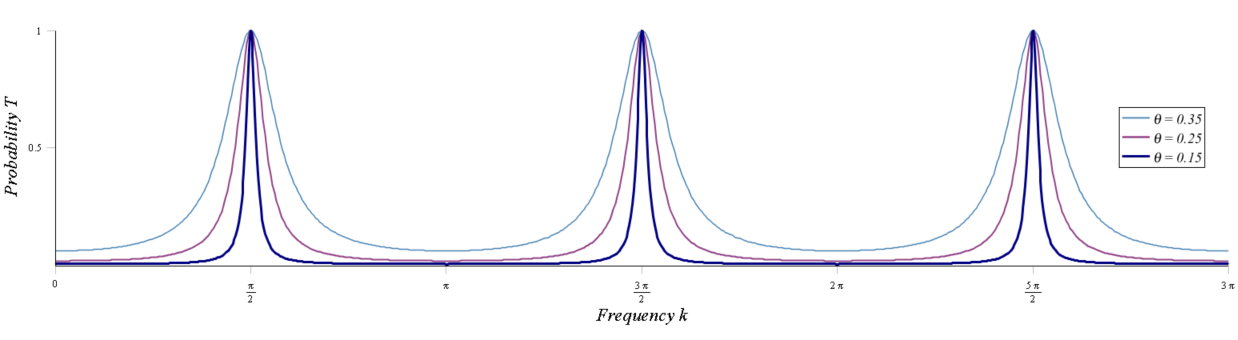}\\
  \caption{Plots of $T_2(\theta,k)$ for different values of $\theta$. }\label{FigN2}
\end{figure}

\begin{figure}[b]
  \centering
  \includegraphics[scale=.38]{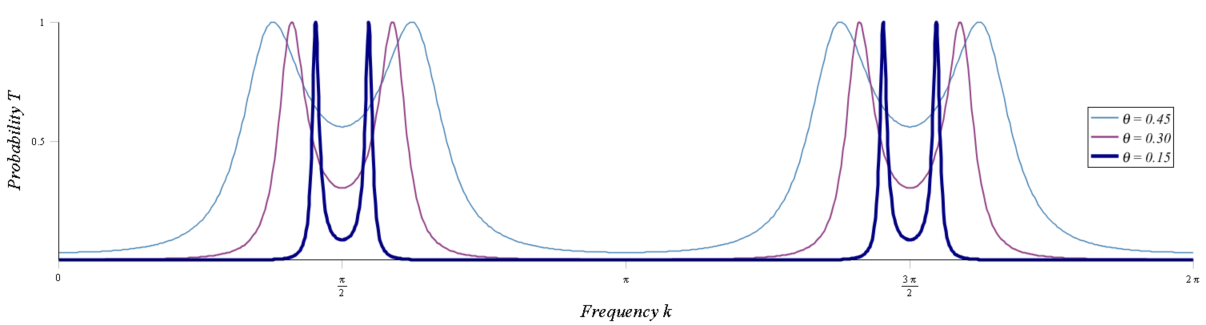}\\
  \caption{Plots of $T_3(\theta,k)$ for different values of $\theta$. }\label{FigN3}
\end{figure}

For a Hamiltonian with three point interactions, the resonance set $\cR_3^\theta$ consists of two points $k_1(\theta)$, $k_2(\theta)$ that are symmetric with respect to $\frac\pi2$. Direct calculation of the transfer matrix $M_3(\theta,k)$ gives
\begin{equation}\label{t3}
\begin{aligned}
   \frac1{t_3(\theta,k)}&=\sinh^2\kappa\cosh\kappa\;(e^{2 i k}+1)^2+\cosh\kappa,\\
   \frac{r_3(\theta,k)}{t_3(\theta,k)}&=-\cosh^2\kappa\sinh\kappa\;(e^{2ik}+1)^2+e^{2ik}\sinh\kappa.
\end{aligned}
\end{equation}
Then the transmission probability
\begin{multline*}
   T_3(\theta,k)=
   \frac1{\cosh^2\kappa\big(16\sinh^2\kappa\cosh^2\kappa\cos^4k-8\sinh^2\kappa\cos^2k+1\big)}
   \\
   =\frac{4\theta^6}{(1+\theta^2)^{2}\big((1-\theta^4)^2\cos^4k
   -2\theta^2(1-\theta^2)^2\cos^2k+\theta^4\big)}
\end{multline*}
has two maxima where it reaches the value $1$ (see Fig.~\ref{FigN3}). The resonances  are easiest to find from the equation $r_3(\theta,k)=0$. Since $(e^{2ik}+1)^2=4 e^{2ik}\cos^2k$, this equation can be written as
\begin{equation*}
  4\cosh^2\kappa\cos^2k=1.
\end{equation*}
Thus,  $\cos k_1(\theta)=\frac{\theta}{1+\theta^2}$, $\cos k_2(\theta)=-\frac{\theta}{1+\theta^2}$, and the resonances are expressed as
\begin{equation*}
  k_1(\theta)=\arccos\frac{\theta}{1+\theta^2},\qquad k_2(\theta)=\pi-\arccos\frac{\theta}{1+\theta^2}.
\end{equation*}
They can also be calculated approximately
\begin{equation*}
  k_1(\theta)=\tfrac\pi2-\theta+\tfrac56\,\theta^3+O(\theta^5),\qquad k_2(\theta)=\tfrac\pi2+\theta-\tfrac56\,\theta^3+O(\theta^5).
\end{equation*}

As the number of point interactions in the Hamiltonian increases, it becomes more complicated to compute the explicit formulas for the transmission probability. For $n=4$, we will find the resonances only using the equation
\begin{equation*}
     \frac{r_4(\theta,k)}{t_4(\theta,k)}=2 e^{2ik}\sinh\kappa\cosh\kappa\;(e^{2ik}+1)(1-2\cosh^2\kappa\cos^2k)=0.
\end{equation*}
The plot of $T_4(\theta,k)$ shown in Fig.~\ref{FigN4} has three peaks at the points
\begin{equation*}
k_1(\theta)=\arccos\frac{\sqrt{2}\theta}{1+\theta^2},\quad k_2=\frac\pi2,\quad k_3(\theta)=\pi-\arccos\frac{\sqrt{2}\theta}{1+\theta^2}.
\end{equation*}
In particular, we have approximate formulas for the moving resonances:
\begin{equation*}
  k_1(\theta)=\tfrac\pi2-\sqrt{2}\theta+\tfrac{2\sqrt{2}}3\,\theta^3+O(\theta^5),\qquad k_3(\theta)=\tfrac\pi2+\sqrt{2}\theta-\tfrac{2\sqrt{2}}3\,\theta^3+O(\theta^5).
\end{equation*}

\begin{figure}[!ht]
  \centering
  \includegraphics[scale=.36]{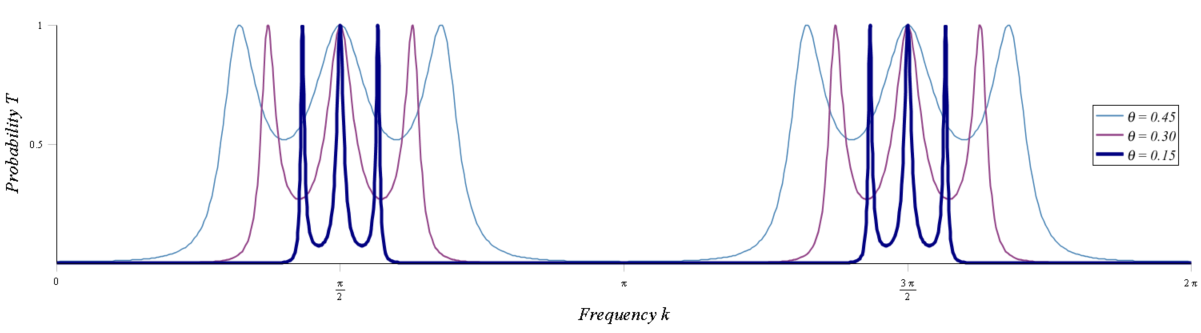}\\
  \caption{Plots of $T_4(\theta,k)$ for different values of $\theta$. }\label{FigN4}
\end{figure}

\subsection{Transmission probability and resonance set in general case}

For every $n \in \mathbb{N}$, there are explicit formulas for the scattering amplitudes $t_n(\theta,k)$ and $r_n(\theta,k)$ and thus for the transmission probability $T_n(\theta,k)$; in this sense, the point interaction Hamiltonian~$\cH$ of~\eqref{OperatorH0} gives a new exactly solvable quantum mechanical model.

\begin{thm}\label{thm:Chebyshev}
For every $n\in\mathbb{N}$ and every $k\in\mathbb{R}$, we have
\begin{equation}\label{tnrnTn}
\begin{aligned}
    \tfrac1{t_n(\theta, k)} & = \tfrac{1+\theta^2}{2\theta}\, U_{n-1}\left(\tfrac{1+\theta^2}{2\theta} \cos k\right)\,e^{i(n-1)k} - U_{n-2}\left(\tfrac{1+\theta^2}{2\theta} \cos k\right) e^{ink},\\
    \tfrac{r_n(\theta, k)}{t_n(\theta, k)} & = \tfrac{1-\theta^2}{2\theta}\, U_{n-1}\left(\tfrac{1+\theta^2}{2\theta} \cos k\right)\,e^{i(n-1)k},
\end{aligned}
\end{equation}
where $U_n$ is the Chebyshev polynomial of second kind. The transmission probability equals
\begin{equation}\label{TnExplicit}
  T_n(\theta,k)=\frac1{1+\frac{(1-\theta^2)^2}{4\theta^2}\: U_{n-1}^2\left(\tfrac{1+\theta^2}{2\theta} \cos k\right)}.
\end{equation}
\end{thm}

Chebyshev polynomials are natural in the context of scattering on finite periodic potential arrays, cf.~\cite{Rorres1974,Griffiths01}.
Recall~\cite{Chihara2011} that the Chebyshev polynomials $U_n$ of second kind are defined by the relation
\begin{equation}\label{eq:ChebRel1}
     \sin \phi\: U_n(\cos \phi)  = \sin (n+1)\phi.
\end{equation}
For $n=0,1,2,3$, these polynomials are explicitly given by
\begin{equation*}
  U_0(\alpha) = 1, \quad U_1(\alpha) = 2\alpha, \quad U_2(\alpha) = 4\alpha^2 - 1,\quad U_3(\alpha) = 8\alpha^3-4\alpha,
\end{equation*}
and for all $n\ge2$, they satisfy the recurrence relation
\begin{equation}\label{ChebRel}
 U_{n}(\alpha) = 2\alpha U_{n-1}(\alpha) - U_{n-2}(\alpha).
\end{equation}
For ease of notation, we will write $a_n$ and $b_n$ instead of the entries $1/t_n$ and $r_n/t_n$ of the transfer matrix $M_n(\theta,k)$, respectively. For $n=2$ and $n=3$, these entries are calculated in \eqref{t2} and \eqref{t3}. Setting $\theta = e^\kappa$ and $\alpha = \cosh\kappa \cos k$, and using the relations
\begin{equation*}
  \sinh^2\kappa=\cosh^2\kappa-1,\quad 2 e^{ik}\cos k=1+e^{2ik},\quad 2 e^{2ik}\cos k= e^{ik}+e^{3ik},
\end{equation*}
we represent the corresponding $a_n$ and $b_n$ as follows:
\begin{alignat*}{2}
    & a_2 = 2\alpha\: e^{ik} \cosh\kappa - e^{2ik}, &  b_2 &= -2 \alpha\: e^{ik}\sinh\kappa ,\\
    & a_3 =  (4\alpha^2 - 1) e^{2ik} \cosh\kappa - 2\alpha e^{3ik},
         \qquad & b_3 &= -(4 \alpha^2 - 1)  e^{2ik} \sinh\kappa.
\end{alignat*}

\begin{proof}
    The proof of \eqref{tnrnTn} is by induction. The above relations show that the statements hold for $n=2$ and $n=3$. Assuming that the formulas hold for the entries $a_{n}$ and $b_{n}$ of the transfer matrix $M_n(\theta,k)$, we then find the entries of $M_{n+1}(\theta,k)$ from the relation $ M_{n+1}(\theta,k) = M(\theta, nk) M_{n}(\theta,k)$, i.e.
    \[
        \begin{bmatrix} \phantom{-}\bar{a}_{n+1} &  -\bar{b}_{n+1} \\ -b_{n+1} &  \phantom{-}a_{n+1} \end{bmatrix}=\begin{bmatrix} \cosh\kappa &\sinh\kappa\: e^{-2ink} \\ \sinh\kappa\: e^{2ink} & \cosh\kappa \end{bmatrix}\begin{bmatrix}  \phantom{-}\bar{a}_n &  -\bar{b}_n \\ -b_n &  \phantom{-}a_n \end{bmatrix}.
    \]
    Replacing $\sinh^2\alpha$ with $\cosh^2\alpha-1$ and $e^{i(n-1)k} + e^{i(n+1)k}$ with $2 e^{ink}\cos k$, we get
    \begin{multline*}
        {a}_{n+1} = \cosh\kappa\, {a}_{n} - \sinh\kappa \:e^{2nik}\, \bar{b}_{n} \\
            = \cosh^2\kappa\, U_{n-1}(\alpha)\,e^{i(n-1)k} - \cosh\kappa\,U_{n-2}(\alpha)  e^{ink}
              + \sinh^2\kappa\, U_{n-1}(\alpha)\,e^{i(n+1)k} \\
            = \big(2\alpha U_{n-1}(\alpha) - U_{n-2}(\alpha)\big) \:e^{ink} \cosh\kappa - U_{n-1}(\alpha)\,e^{i(n+1)k}\\
            = U_n(\alpha)\: e^{ink} \cosh\kappa - U_{n-1}(\alpha)\,e^{i(n+1)k},
    \end{multline*}
    and, similarly,
    \begin{multline*}
        -b_{n+1} = \sinh\kappa\: e^{2nik}\, \bar{a}_{n} - \cosh\kappa\, {b}_{n}
            = \sinh\kappa\cosh\kappa\, U_{n-1}(\alpha)\,e^{i(n+1)k} \\- \sinh\kappa\,U_{n-2}(\alpha) \, e^{ink}
           + \sinh\kappa\cosh\kappa\, U_{n-1}(\alpha)\,e^{i(n-1)k}
            \\= \sinh\kappa\: (2\alpha U_{n-1}(\alpha) - U_{n-2}(\alpha)) \: e^{ink}
            =  \sinh\kappa\: U_n(\alpha) \: e^{ink}
    \end{multline*}
    due to the recurrent relation~\eqref{ChebRel} satisfied by $U_n$, thus proving \eqref{tnrnTn}.

    Since $|a-b e^{ik}|^2=a^2+b^2-2ab\cos k$ for any two real numbers $a$ and $b$, we have
    \begin{multline*}
      \frac1{|t_n(\theta, k)|^2}  = \left|\cosh\kappa\: U_{n-1}(\alpha)- U_{n-2}(\alpha) e^{ik}\right|^2
      \\=
      \cosh^2\kappa\: U_{n-1}^2(\alpha)+ U_{n-2}^2(\alpha)-2\alpha U_{n-1}(\alpha)U_{n-2}(\alpha).
    \end{multline*}
    The recurrent relation \eqref{ChebRel} allows us to rewrite the above as
    \begin{equation}\label{TnStep}
      \frac1{|t_n(\theta, k)|^2}  =
      \cosh^2\kappa\: U_{n-1}^2(\alpha)- U_{n-2}(\alpha)U_{n}(\alpha).
    \end{equation}
    Next, we note that $U_{n-2}U_{n}= U_{n-1}^2-1$; this follows from the defining relation $\sin\phi\: U_n(\cos \phi)=\sin(n+1)\phi$ for the Chebyshev polynomials and straightforward transformations:
    \begin{multline*}
      U_{n-2}(\cos \phi)U_{n}(\cos \phi)= \frac{\sin(n-1)\phi\:\sin(n+1)\phi}{\sin^2\phi}\\=\frac{\cos2\phi-\cos2n\phi}{2\sin^2\phi}
      =\frac{\sin^2n\phi -\sin^2 \phi}{\sin^2\phi}=U_{n-1}^2(\cos \phi)-1.
    \end{multline*}
    Finally, we conclude from \eqref{TnStep} that
    \begin{multline*}
      \frac1{|t_n(\theta, k)|^2}  = (\cosh^2\kappa-1)\: U_{n-1}^2(\alpha)+1\\=\sinh^2 \kappa\: U_{n-1}^2(\alpha)+1=\tfrac{(1-\theta^2)^2}{4\theta^2}\:U_{n-1}^2\left(\tfrac{\theta^2+1}{2\theta} \cos k\right)+1.
    \end{multline*}
    The proof is complete.
\end{proof}


\begin{thm}
    For every integer $n\ge 2$ and $\theta\in(0,1)$, the resonance set $\cR_{\theta,n} $ for the transmission probability $T_n(\theta,k)$ contains $n-1$ points $k_1(\theta),\dots, k_{n-1}(\theta)$, where $k_j(\theta)$ is the root of the equation
    \begin{equation}\label{CosK}
         \cos k_j = \frac{2\theta}{\theta^2+1}\cos \frac{\pi j}{n}
    \end{equation}
on the interval $(0,\pi)$, $j =1,2,\dots,n-1$.    At these points, $T_n(\theta,k)$ reaches its global maximum, which is equal to $1$.
\end{thm}

\begin{proof}
  It follows from~\eqref{eq:scat-identity} that, for all $k\in \Real$, $T_n(\theta,k) \le 1$ and that $T_n(\theta,k) =1$ if and only if $r_n(\theta,k)=0$.
  By Theorem~\ref{thm:Chebyshev}, the equation $r_n(\theta,k)=0$ is equivalent to
  \begin{equation*}
   U_{n-1}\left(\tfrac{\theta^2+1}{2\theta} \cos k\right)=0.
  \end{equation*}
  It is known~\cite{Chihara2011} that the polynomial $U_{n-1}(\alpha)$ has $n-1$ simple roots
\begin{equation*}
  \alpha_j=\cos\frac{\pi j}{n}, \quad j=1,2,\dots,n-1.
\end{equation*}
 Hence, the roots $k_j(\theta)$ of \eqref{CosK} belong to $\cR_{\theta,n}$ and $T_n(\theta,k_j(\theta))=1$.
 We next prove that there are no other transmission resonances. Observe that if $\beta_j$ and $\beta_{j+1}$ are two consecutive critical points of $U_{n-1}$, then $U_{n-1}(\beta_j)U_{n-1}(\beta_{j+1})<0$ and $U_{n-1}$ is strictly monotone on $[\beta_j, \beta_{j+1}]$. Since the derivative $\frac{\partial T_n}{\partial k}$ has zeros at the points~$k\in[0,\pi)$ that are roots of the equation
\begin{equation*}
   U_{n-1}\left(\tfrac{\theta^2+1}{2\theta} \cos k\right)U_{n-1}'\left(\tfrac{\theta^2+1}{2\theta} \cos k\right)\sin k=0,
\end{equation*}
we conclude that all critical points of $T_n(\theta,\,\cdot\,)$ other than $k_j(\theta)$, $j=1,2,\dots,n-1$, are points of its local minimum.
\end{proof}


\begin{figure}[b]
  \centering
  \includegraphics[scale=.4]{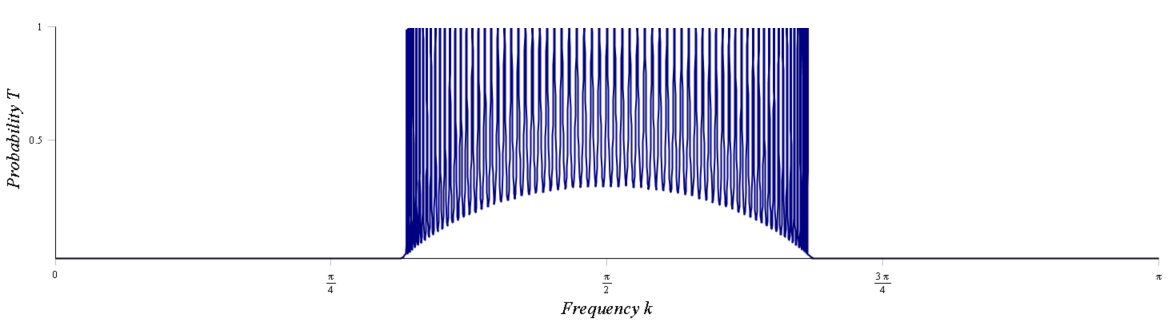}\\
  \caption{Plot of $T_{80}(\theta,k)$ for $\theta=0.3$.}\label{fig:n80}
\end{figure}


For a fixed $\theta$, the peaks of the transmission probabilities $T_n(\theta,k)$, for all $n$, belong to the interval of~$k$ with
$\frac{\theta^2 + 1}{2\theta}|\cos k| < 1$, i.e., to the interval
\begin{equation*}
    I_\theta = \left( \tfrac{\pi}2 - \arcsin{\tfrac{2\theta}{\theta^2 +1}},
                            \tfrac{\pi}2 + \arcsin{\tfrac{2\theta}{\theta^2 +1}}\right).
\end{equation*}
The plot of $T_n(\theta,k)$ for a fixed $\theta$ and $n=80$ is given in Fig.~\ref{fig:n80}; note the dense array of transmission peaks over the range $I_\theta$.

The explicit representation \eqref{TnExplicit} of the transmission probability allows us to analyze its behaviour as $\theta\to0$ more precisely.

\begin{cor}
  For any $k\in(0,\pi)$ other than $\frac\pi2$, the following estimate holds:
  \begin{equation*}
    T_n(\theta,k)\le \frac{c\:\theta^{2n}}{|\cos k|^{2(n-1)}},\qquad \text{as } \theta\to 0,
  \end{equation*}
  where  $c$ is independent of $\theta$ and $k$.
\end{cor}

\begin{proof}
  For large $\alpha$, we have $U_{n-1}(\alpha)=2^{n-1}\alpha^{n-1}+O(\alpha^{n-3})$. Hence,
  $$|U_{n-1}(\alpha)|\ge 2^{n-2}|\alpha|^{n-1}$$ as $|\alpha|\to \infty$. If $d=|\cos k|>0$, then
  \begin{equation*}
    \left|\frac{\theta^2+1}{2\theta} \cos k\right|=\frac d{2\theta}+O(\theta),\quad\text{as } \theta\to 0,
  \end{equation*}
  and therefore
  \begin{equation*}
    \left|U_{n-1}\left(\tfrac{\theta^2+1}{2\theta} \cos k\right)\right|\ge \frac12 \frac{d^{n-1}}{\theta^{n-1}}.
  \end{equation*}
  Finally, we conclude from \eqref{TnExplicit} that
  \begin{equation*}
  T_n(\theta,k)\le\frac1{1+ c_1 d^{2(n-1)}/\theta^{2n}}\le\frac{c\:\theta^{2n}}{d^{2(n-1)}},
  \end{equation*}
 as claimed.
\end{proof}

Another interesting observation is that the infimum of $T_n(\theta,k)$ over~$n$ coincides with limit inferior, is positive for $k\in I_\theta$, and takes the following explicit form.


\begin{figure}[b]
  \centering
  \includegraphics[scale=.4]{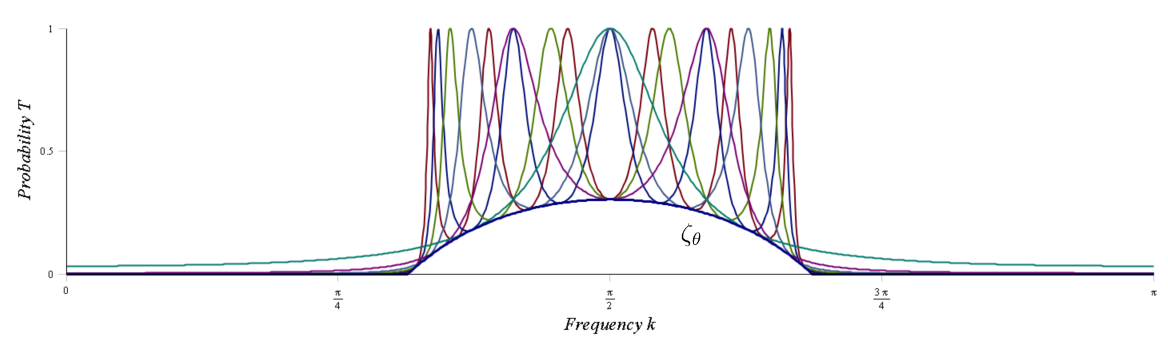}\\
  \caption{Envelope $\zeta_\theta(k)$ and plots of $T_n(\theta,k)$ for $\theta=0.3$ and several different values of $n$.}\label{fig:envelop}
\end{figure}



\begin{cor}
For a fixed $\theta\in(0,1)$ and all $k \in [0,\pi)$, it holds
\begin{equation*}
     \inf_{n\in \mathbb{N}} T_n(\theta,k) = \begin{cases}
        1 - \dfrac{(\theta^2 -1)^2}{(\theta^2 + 1)^2 \sin^2 k}, \qquad &k \in I_\theta; \\
        0, & k \in [0,\pi)\setminus I_\theta.
    \end{cases}
\end{equation*}
\end{cor}

\begin{proof}
    In view of the explicit formula~\eqref{TnExplicit} for $T_n(\theta,k)$, it suffices, for each fixed $\theta\in(0,1)$ and $k\in [0,\pi)$, to find the supremum of $U^2_{n-1}(\cosh\kappa \cos k)$, where $\kappa$ satisfies $\theta = e^\kappa$. Assume that $\cosh\kappa\: |\cos k| <1$ and take $\phi\in(0,\pi)$ such that $\cosh\kappa\: \cos k = \cos\phi$. Then, by the defining property~\eqref{eq:ChebRel1} of the Chebyshev polynomials of second kind,
    \[
        \sup_{n\in\mathbb{N}} U^2_{n-1}(\cos\phi)=\sup_{n\in\mathbb{N}} \frac{\sin^2 (n\phi)}{\sin^2 \phi} = \frac1{\sin^2\phi} = \frac1{1 - \cos^2\phi},
    \]
    impying that
    \begin{align*}
        \inf_{n\in\mathbb{N}} T_n(\theta,k) &=  \bigl(1 + \sinh^2\kappa /(1 - \cosh^2\kappa\: \cos^2k)\bigr)^{-1} \\
            &= \frac{1 - \cosh^2\kappa\:\cos^2k}{\cosh^2\kappa - \cosh^2\kappa\: \cos^2k}
             = 1 - \frac{\sinh^2\kappa}{\cosh^2\kappa\: \sin^2 k}.
    \end{align*}
    If $\cosh\kappa |\cos k| \ge 1$, then $U^2_{n-1}(\cosh\kappa \cos k) \ge U^2_{n-1}(1) \to \infty$ as $n\to\infty$, showing that $T_n(\theta,k)$ decays to zero. The proof is complete.
\end{proof}

One can also prove that the curve  $\zeta_\theta(k)=\inf_{n\in\mathbb{N}} T_n(\theta,k)$ is the envelope of the family~$T_n(\theta,k)$ over the interval $I_\theta$, see Fig.~\ref{fig:envelop}.

Finally, we demonstrate that transmission prevails for $\theta$ close to~$1$; thus the $\delta'_\theta$-comb becomes completely penetrable as $\theta\to1$, see Fig.~\ref{FigThetaGoes1}.

\begin{cor}\label{cor:upper_bound}
    As $\theta$ tends to $1$, the functions $T_n(\theta,k)$ converge to $1$ uniformly in $k \in (0,\pi]$; moreover, the inequality
    \begin{equation*}
      1-T_n(\theta,k)\le 4n^2 (1-\theta)^2
    \end{equation*}
    holds for all $\theta$ close to $1$.
\end{cor}

\begin{proof}
  First, we note that
  \begin{equation*}
  \max_{|\phi|\le \pi}|U_{n}(\cos \phi)|=\max_{|\phi|\le \pi}\frac{|\sin(n+1)\phi|}{|\sin\phi|}=n+1.
  \end{equation*}
  Since $U_{n}\left(\tfrac{\theta^2+1}{2\theta} \cos k\right)\to U_{n}(\cos k)$ as $\theta\to 1$, we can estimate
  \begin{equation*}
    \left|U_{n}\left(\tfrac{\theta^2+1}{2\theta} \cos k\right)\right|\le 2(n+1)
  \end{equation*}
  for all $\theta$ close to $1$.
  Set $\beta_n(\theta,k)=\frac{(1-\theta^2)^2}{4\theta^2}\: U_{n-1}^2\left(\tfrac{\theta^2+1}{2\theta} \cos k\right)$. Then
  \begin{equation*}
   \beta_n(\theta,k)\le 4n^2 (1-\theta)^2,\qquad\text{as }\theta\to 1.
  \end{equation*}
 We conclude from this that
  \begin{equation*}
  1-T_n(\theta,k)=1-\frac{1}{1+\beta_n(\theta,k)}=\frac{\beta_n(\theta,k)}{1+\beta_n(\theta,k)}\le \beta_n(\theta,k)\le 4n^2 (1-\theta)^2,
\end{equation*}
which completes the proof.
\end{proof}


\section{Scattering from locally periodic smooth resonant potentials}\label{sec:example}



As shown above, the exactly solvable model Hamiltonians~$\cH$ with $\delta'_\theta$-comb potentials exhibit distinctive and useful transmission properties, making them valuable for understanding quantum filtering and tunnelling phenomena. Assume that~$V$ is a resonant potential with associated parameter $\theta$  as defined in~\eqref{Theta}. Then the family of scaled Hamiltonians~$\cH_\eps$ in~\eqref{Heps} converges, as $\eps\to0$, to the model Hamiltonian~$\cH$ of~\eqref{OperatorH0} defined by~\eqref{eq:cHdef}--\eqref{eq:cHinterface}, in the norm resolvent sense~\cite{GolovatyHryniv2010}. In this section, we pursue two goals: firstly, we construct, for every $\theta>0$, a compactly supported potential $V$ having a zero-energy resonance with this parameter; secondly, we prove that the transmission probability for $\cH_\eps$ converges as $\eps\to0$ to that of $\cH$ (cf.~\cite{GolovatyManko2009}). This demonstrates that the non-trivial band-limited behaviour of the Hamitonians with $\delta'_\theta$-comb potentials persists under realistic approximations.

\begin{lem}
  For any $\theta>0$, there exists a potential $V \in L^1(\Real)$ with compact support such that the Schr\"{o}dinger operator $-\frac{d^2}{dx^2}+V(x)$ has zero-energy resonance with a half-bound state $u$ satisfying 
  \begin{equation}\label{UforGivenTheta}
   \frac{u(+\infty)}{u(-\infty)}=\theta.
  \end{equation}
\end{lem}
\begin{proof}
  Let $u$ be a smooth positive function on the line such that $u(x)=1$ for $x<-1$ and $u(x)=\theta$ for $x>1$ and set 
  \begin{equation}\label{VasU}
      V(x)=\frac{u''(x)}{u(x)}.
  \end{equation}
By construction, $u$ solves the equation 
\begin{equation}\label{EqnYV}
 -y''+V(x)y=0
\end{equation}
and is bounded on the line, so that $V$ is a resonant potential with half-bound state~$u$. In addition, the support of $V$ lies in $[-1,1]$ and \eqref{UforGivenTheta} holds.
\end{proof}

Let us consider the scattering problem 
\begin{equation}\label{SingleDipole}
  -y''+\frac1{\eps^2}\,V\left(\frac{x-x_0}{\eps}\right)y=k^2 y
\end{equation}
on a single dipole localized at the point $x=x_0$.

\begin{thm}\label{thm:TransMatrixV}
  Assume that $V$ is a resonant potential, $\supp V\subset [-1,1]$, and that $u$ is the corresponding half-bound state.  Then the transfer matrix of problem \eqref{SingleDipole} has the form
\begin{equation}\label{TransMatrixMeps}
  M_\eps(V; x_0,k)=\dfrac1{2\theta}\begin{bmatrix}
        (\theta^2+1+i\eps k \eta)\,e^{-2i\eps k} & (\theta^2-1-i\eps k \eta)\,e^{-2ikx_0} \\
        (\theta^2-1+i\eps k \eta)\,e^{2ikx_0} &   (\theta^2+1-i\eps k \eta)\,e^{2i\eps k}
  \end{bmatrix},
\end{equation}
where $\theta=u(1)/u(-1)$ and $\eta=\theta^2\int_{-1}^1 u^{-2}(x)\,dx$.
\end{thm}

\begin{proof}
Since $\supp V \subset [-1,1]$, the half-bound state $u$ is constant outside the interval $(-1,1)$. Without loss of generality, we assume that
\begin{equation}\label{UatEnds}
    u(-1)=1,\qquad u(1)=\theta;
\end{equation}
then $u(\pm\infty) = u(\pm1)$, and $\theta$ is the parameter associated with the resonant potential~$V$ via~\eqref{Theta}. 

A solution of~\eqref{EqnYV} that is linearly independent of $u$ can be taken
\begin{equation}\label{SolutionV}
  v(x)=u(x)\int_{-1}^x \frac{dt}{u^2(t)}.
\end{equation}
Then $uv'-u'v=1$, so that on account of~\eqref{UatEnds} and the equalities $u'(\pm1)=0$, we conclude that  
\begin{equation}\label{VatEnds}
    v(-1)=0,\quad v'(-1)=1, \quad v(1)=\theta \int_{-1}^1u^{-2}(x)\,dx, \quad v'(1)=\theta^{-1}.
\end{equation}

We are looking for a solution of \eqref{SingleDipole} of the form
\begin{equation*}
  y(x)=
  \begin{cases}
    \alpha_1e^{ikx}+\alpha_2e^{-ikx}&\text{for }x<x_0-\eps,\\
    c_1u\left(\tfrac{x-x_0}{\eps}\right)+c_2 v\left(\tfrac{x-x_0}{\eps}\right)&\text{for }x_0-\eps<x<x_0+\eps,\\
    \beta_1e^{ikx}+\beta_2e^{-ikx}&\text{for }x>x_0+\eps.
  \end{cases}
\end{equation*}
The transfer matrix $M_\eps(V ; x_0, k)$ maps the vector $[\alpha_1,\alpha_2]^\top$ into $[\beta_1,\beta_2]^\top$, thus we need to express the constants $\beta_j$ in terms of $\alpha_j$. Since the solution $y$ must be continuously differentiable at $x=x_0-\eps$, the equalities $y(x_0 - \eps -0) = y(x_0 - \eps +0)$ and $y'(x_0 - \eps -0) = y'(x_0 - \eps +0)$ yield
\begin{equation*}
  \begin{bmatrix}
      e^{ik(x_0-\eps)} & e^{-ik(x_0-\eps)} \\ ik e^{ik(x_0-\eps)} & - ik e^{-ik(x_0-\eps)}
  \end{bmatrix}
  \begin{bmatrix}
      \alpha_1 \\ \alpha_2
  \end{bmatrix}
    = 
   \begin{bmatrix}
       u(-1) & v(-1) \\ \eps^{-1} u'(-1) & \eps^{-1} v'(-1)
   \end{bmatrix} 
   \begin{bmatrix}
      c_1 \\ c_2
  \end{bmatrix};
\end{equation*}
similar continuity conditions at $x = x_0 + \eps$ produce
\begin{equation*}
  \begin{bmatrix}
       u(1) & v(1) \\ \eps^{-1} u'(1) & \eps^{-1} v'(1)
   \end{bmatrix} 
   \begin{bmatrix}
      c_1 \\ c_2
  \end{bmatrix}
  =
  \begin{bmatrix}
      e^{ik(x_0+\eps)} & e^{-ik(x_0+\eps)} \\ ik e^{ik(x_0+\eps)} & - ik e^{-ik(x_0+\eps)}
  \end{bmatrix}
  \begin{bmatrix}
      \beta_1 \\ \beta_2
  \end{bmatrix}.
\end{equation*}
Using the values of $u$ and $v$ at the points $\pm1$ listed in \eqref{UatEnds} and \eqref{VatEnds} and introducing $\eta$ as suggested, we conclude that 
\begin{equation*}
  \begin{bmatrix}
      e^{ik(x_0+\eps)} & e^{-ik(x_0+\eps)} \\  e^{ik(x_0+\eps)} & - e^{-ik(x_0+\eps)}
  \end{bmatrix}
  \begin{bmatrix}
      \beta_1 \\ \beta_2
  \end{bmatrix}
  =
  \begin{bmatrix}
    \theta & ik\eps\eta/\theta \\ 0 & 1/\theta
  \end{bmatrix}
  \begin{bmatrix}
      e^{ik(x_0-\eps)} & e^{-ik(x_0-\eps)} \\  e^{ik(x_0-\eps)} & -  e^{-ik(x_0-\eps)}
  \end{bmatrix}
  \begin{bmatrix}
      \alpha_1 \\ \alpha_2
  \end{bmatrix}.
\end{equation*}
Solving for the vector $[\beta_1,\beta_2]^\top$, we derive the transfer matrix $M_\eps(V ; x_0, k)$ as in~\eqref{TransMatrixMeps}.
\end{proof}

\begin{thm}
  Let $T_{n,\eps}(V;k)=|t_{n,\eps}(V;k)|^2$ be the probability of a particle tunneling through the array of localized dipoles
\begin{equation}\label{PotentialVeps}
   \frac1{\eps^2}\sum_{j=0}^{n-1}V\left(\frac{x-j}\eps\right),
\end{equation}
with potential $V$ as in Theorem~\ref{thm:TransMatrixV}.  Then $T_{n,\eps}(V;k)$ converges as $\eps\to 0$ to the transmission probability $T_n(\theta,k)$ for the $\delta_\theta'$-comb
\begin{equation*}
  \sum_{j=0}^{n-1}\delta'_{\theta}(x-j).
\end{equation*}
\end{thm}
\begin{proof}
Comparing the transfer matrices in \eqref{TransMatrixMeps} and \eqref{TransferMatrix1}, we see that 
  \begin{equation}\label{MeMeps}
    \| M_\eps(V; x_0, k)-M(\theta,kx_0)\|\le c_1|k|\eps
  \end{equation}
with a constant $c_1$ independent of~$k$. 
The transfer matrix for the potential~$V_\eps$ is equal to
\begin{equation}\label{MneMnEps}
  M_{n,\eps}(V;k)=M_\eps(V; n-1,k)\cdots M_\eps(V; 1,k)M_\eps(V; 0,k),
\end{equation}
and using \eqref{MeMeps} in \eqref{TransferMatrixN} and \eqref{MneMnEps}, one concludes that
  \begin{equation*}
    \| M_{n,\eps}(V;k)-M_n(\theta,k)\|\le c_2 n |k|\eps.
  \end{equation*}
Therefore, $|t^{-1}_{n,\eps}(V;k)-t_n^{-1}(\theta,k)|\le c_2 n |k|\eps$, so that, for every fixed $k\in\Real$,
\begin{equation*}
    T_{n,\eps}(V;k)\to T_n(\theta,k),
\end{equation*}
as $\eps\to 0$. The proof is complete. 
\end{proof}

Finally, we present an explicit example of a potential~\eqref{PotentialVeps} with a suitable resonant~$V$. 
Consider the potential $V_\theta$ defined via
  \begin{equation*}
    V_\theta(x)= \frac{48x^2+6(1-\theta)x-16}{(x+1)^2\big(4x^2-(\theta+7)x+2(\theta+1)\big)+4}, \quad |x|\le 1,
  \end{equation*}
and $V_\theta(x)=0$ for $|x| > 1$. Each potential in this family with $\theta\ne0$ is resonant and has a half-bound states given by
\begin{equation*}
  u_\theta(x)=
  \begin{cases}
    \hskip1cm1& \text{for }x<-1,\\
    \frac14(x+1)^2\big(4x^2-(\theta+7)x+2(\theta+1)\big)+1& \text{for }|x|\le 1,\\
    \hskip1cm\theta &\text{for }x>1.
  \end{cases}
\end{equation*}
For each non-zero $\theta$, one can determine the corresponding parameter $\eta$ as described in Theorem~\ref{thm:TransMatrixV}; for instance, when $\theta = 0.2$, we find $\eta = 0.21295(7)$. Using the matrix $ M_\eps(V_\theta; x_0,k)$ in~\eqref{TransMatrixMeps}, we compute the transmission probability $T_{n,\eps}(V_\theta;k)$. As shown in Figures~\ref{fig:TepsN=2}--\ref{fig:TepsN=4}, the transmission probability exhibits a pronounced peak structure, consistent with the predictions of the exactly solvable model. Even for moderately small $\eps$, the potentials \eqref{PotentialVeps} with resonant $V_\theta$ demonstrate clear band-pass filtering properties.


\begin{figure}[!ht]
  \centering
  \includegraphics[scale=.5]{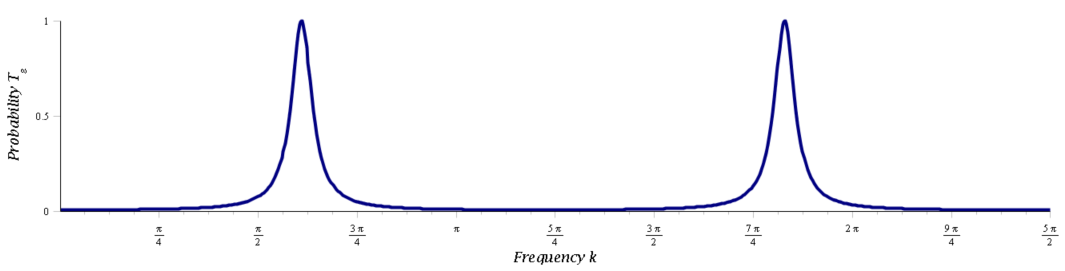}\\
  \caption{Plot of $T_{2,\eps}(V_\theta;k)$ for $\theta=0.2$ and $\eps=0.1$.}\label{fig:TepsN=2}
\end{figure}



\begin{figure}[!ht]
  \centering
  \includegraphics[scale=.5]{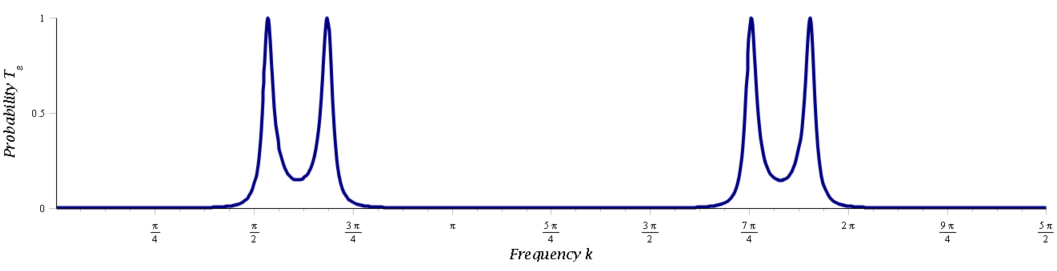}\\
  \caption{Plot of $T_{3,\eps}(V_\theta;k)$ for $\theta=0.2$ and $\eps=0.1$.}\label{fig:TepsN=3}
\end{figure}



\begin{figure}[!ht]
  \centering
  \includegraphics[scale=.5]{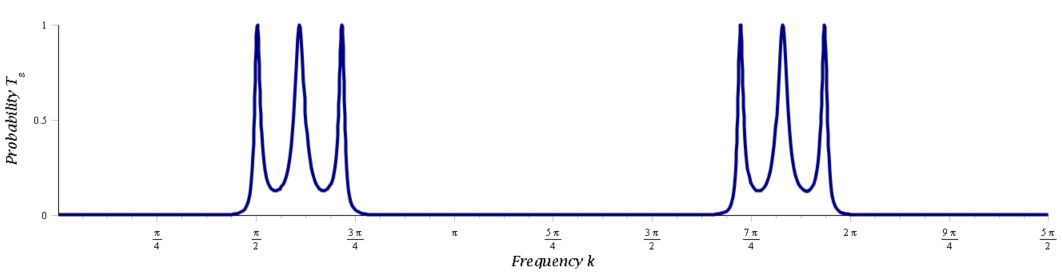}\\
  \caption{Plot of $T_{4,\eps}(V_\theta;k)$ for $\theta=0.2$ and $\eps=0.1$.}\label{fig:TepsN=4}
\end{figure}



\end{document}